\documentclass[runningheads]{llncs}
\usepackage{graphicx}
\usepackage{amsfonts}
\usepackage{algorithm}
\usepackage{algpseudocode}
\usepackage{amsmath}
\usepackage{amssymb}
\usepackage{dsfont}
\usepackage{bm}
\usepackage{xcolor}
\usepackage{colortbl}
\usepackage{color}
\usepackage{graphicx}
\usepackage{multirow}
\usepackage{pifont}
\usepackage{todonotes}
\usepackage{mathtools}
\usepackage{hyperref}
\usepackage{cleveref}
\usepackage{caption}
\usepackage{subcaption}

\usepackage{verbatim}
\usepackage{xspace} %

\usepackage{enumerate}
\usepackage{enumitem}
\allowdisplaybreaks

\newcommand{\norm}[1]{\|#1\|}

\newcommand{\R}{\ensuremath{\mathbb{R}}}

\newcommand{\Prob}{\ensuremath{\mathbb{P}}}
\newcommand{\EE}{\mathbb{E}}
\newcommand{\E}{\mathbb{E}}

\def\<#1,#2>{\left\langle #1,#2 \right\rangle}

\def\batchbound{M}
\newcommand{\mathbbm}[1]{\text{\usefont{U}{bbm}{m}{n}#1}}
\newcommand{\indiacc}[1]{\mathbbm{1}_{\{#1\}}}
\newcommand{\tvnorm}[1]{\| #1 \|_{\operatorname{TV}}}
\newcommand{\PEcoupling}[2]{\tilde{\mathbb{E}}_{#1,#2}}
\newcommand{\PPcoupling}[2]{\tilde{\mathbb{P}}_{#1,#2}}
\newcommand{\PP}{\mathbb{P}}
\def\metricz{\mathsf{d}_{\Zset}}
\def\rset{\mathbb{R}}

\def\cmax{\sigma}
\def\dmax{\delta}

\def\nset{\ensuremath{\mathbb{N}}}

\def\Zset{\mathsf{Z}}
\def\Zsigma{\mathcal{Z}}
\def\MKQ{\mathrm{Q}}

\def\taumix{\tau}
\newcommand{\dobrush}{\mathsf{\Delta}}
\def\nbiter{N}

\newtheorem{assumption}{Assumption}
\begin{document}
\title{Accelerated Stochastic Gradient Method with Applications to Consensus Problem in Markov-Varying Networks \thanks{The research was supported by Russian Science Foundation (project No. 23-11-00229).}}
\titlerunning{Accelerated SGD for Consensus in Markov-Varying Networks}
\author{Vladimir Solodkin\inst{1,2} \and
Savelii Chezhegov\inst{1,2} \and
Ruslan Nazikov\inst{1,2} \and
Aleksandr Beznosikov\inst{1,2,3} \and 
Alexander Gasnikov\inst{3,1,2}}
\authorrunning{V. Solodkin
S. Chezhegov
R. Nazikov
A. Beznosikov
A. Gasnikov}
\institute{Moscow Institute of Physics and Technology, Moscow, Russian Federation \and
Institute for Information Transmission Problems RAS, Moscow, Russian Federation
 \and
Innopolis University, Innopolis, Russian Federation}
\maketitle              % typeset the header % We propose a variation of accelerated stochastic gradient descent in the Markovian noise setting. Our analysis relies on
\begin{abstract}
Stochastic optimization is a vital field in the realm of mathematical optimization, finding applications in diverse areas ranging from operations research to machine learning. In this paper, we introduce a novel first-order optimization algorithm designed for scenarios where Markovian noise is present, incorporating Nesterov acceleration for enhanced efficiency. The convergence analysis is performed using an assumption on noise depending on the distance to the solution. We also delve into the consensus problem over Markov-varying networks, exploring how this algorithm can be applied to achieve agreement among multiple agents with differing objectives during changes in the communication system. To show the performance of our method on the problem above, we conduct experiments to demonstrate the superiority over the classic approach.
% Through a blend of theoretical analysis and numerical experiments, we illustrate the effectiveness of our approach across various optimization tasks. The combination of Markovian noise modeling and Nesterov acceleration offers a promising avenue for advancing the field of stochastic optimization, presenting opportunities for improved solutions in complex and noisy optimization landscapes.
\end{abstract}

\section{Introduction}
Stochastic optimization encompasses a suite of methodologies aimed at minimizing or maximizing an objective function in the presence of randomness. These methods have evolved into indispensable tools across a spectrum of disciplines including science, engineering, business, computer science, and statistics. Applications are diverse, ranging from refining the placement of acoustic sensors on a beam through simulations, to determining optimal release times for reservoir water to maximize hydroelectric power generation, to fine-tuning the parameters of statistical models based on given datasets. The introduction of randomness typically occurs through the cost function or the constraint set. While the term "stochastic optimization" may encompass any optimization approach that incorporates randomness within certain communities, our focus here is on scenarios where the objective function is stochastic.\\
As with deterministic optimization, no universal solution method generally excels across all problems. Structural assumptions play a pivotal role in making problems tractable. Given that solution methodologies are intricately linked to problem structures, our analysis relies heavily on problem type, with a detailed exposition of associated solution approaches.\\\\
\noindent \textbf{Related work.}
% There has been extensive research on accelerating gradient descent. 
A considerable body of research has documented substantial advancements achieved by accelerating gradient descent in a Nesterov manner \cite{nesterov_accelerated}. Building upon this foundation, Nesterov-accelerated stochastic gradient descent \cite{hu2009accelerated,lan2012optimal} emerged as a powerful tool for optimizing different objectives in stochastic settings. In the earlier works \cite{lan2008,devolder}, the proof of convergence was done using an assumption on bounded variance, which significantly narrows the application perspective. Later, \cite{schmidt2013fast} succeeded in relaxing this assumption to strong growth condition, which partially solved the aforementioned problem. At the same time, several papers delved into applying acceleration to specific stochastic cases, e.g., coordinate descent \cite{doi:10.1137/100802001}, heavy tailed noise \cite{wang2021convergence}, distributed learning \cite{Qu_2020}. However, all of these works investigate i.i.d. noise setup, while a more general case could be considered. \\
% A considerable body of research has documented substantial advancements achieved by accelerated methods in stochastic optimization employing stochastic gradient oracles with independent and identically distributed noise \cite{robbinsmonro,nemirovskiyudin78,nemirovskiyudin83,gorbunov2020iid}. These methods have been extensively explored in theoretical frameworks with demonstrated practical successes \cite{kingma2014adam,pmlr-v28-sutskever13}. 
% A detailed exploration of finite-time analysis for first-order methods in i.i.d. noise settings has been undertaken by various researchers, as reviewed in \cite{lan20} and the associated references.\\
As of late, there has been an emergence of scholarly works aimed at addressing the existing gap in the analysis of Markovian noise configuration. Nonetheless, it is noteworthy that this domain continues to be a dynamically evolving field of study.
Specifically, \cite{duchi2012ergodic} examined a variant of the Ergodic Mirror Descent algorithm yielding optimal convergence rates for smooth and nonconvex problems. More recently, \cite{dorfman2022adapting} proposed a random batch size algorithm tailored for nonconvex optimization within a compact domain. 
In the Markovian noise domain, the finite-time analysis of non-accelerated SGD-type algorithms has been investigated in \cite{sun2018markov} and \cite{doan23}. However, \cite{sun2018markov} relies heavily on the assumption of a bounded domain and uniformly bounded stochastic gradient oracles and \cite{doan23} achieves only suboptimal dependence on initial conditions for strongly convex problems when employing SGD. In the exploration of accelerated SGD in the presence of Markovian noise, \cite{doan2020convergence} achieved an optimal rate of initial condition forgetting, but suboptimal variance terms. Recently, \cite{beznosikov2023first} proposed the accelerated version of SGD achieving linear dependence on the mixing time.\\
The aforementioned studies predominantly address general Markovian noise optimization. Recently, a surge of papers has emerged, focusing on the specialized scenario of distributed optimization \cite {pmlr-v162-sun22b,even2023stochastic}. \cite{wang2022stability} investigates the generalization and stability of Markov SGD with specific emphasis on excess variance guarantees. Simultaneously, specific results such as those from \cite{nagaraj2020least} offer lower bounds for particular finite-sum problems within the Markovian setting.\\\\
\noindent \textbf{Our contributions.} We present the analysis of an accelerated version of SGD in the Markovian noise setting under the assumption of a gradient estimator bounded by the distance to the optimum. We obtain sharp convergence rate and optimal dependence in terms of the mixing time of the underlying Markov chain. Moreover, for $k = 1$ Markovian scheme reduces to a classical $i.i.d.$ noise setup. To the best of our knowledge, analysis in this case (even for $i.i.d.$ stochasticity) under suggested assumptions has not been presented in the literature before. To show the practicality of our method, we perform numerical experiments on the consensus search problem on time-varying networks and show a better convergence rate compared to classical approaches for solving this problem.

\subsection{Technical Preliminaries}
Let $(\Zset,\metricz)$ be a complete separable metric space endowed with its Borel $\sigma$-field $\Zsigma$. Let $(\Zset^{\nset},\Zsigma^{\otimes \nset})$ be the corresponding canonical process. Consider the Markov kernel $\MKQ$ defined on $\Zset \times \Zsigma$, and denote by $\PP_{\xi}$ and $\E_{\xi}$ the corresponding probability distribution and the expected value with initial distribution $\xi$. Without loss of generality, we assume that $(Z_k)_{k \in \nset}$ is the corresponding canonical process. By construction, for any $A \in \Zsigma$, it holds that $\PP_{\xi}(Z_k \in A | Z_{k-1}) = \MKQ(Z_{k-1},A)$, $\PP_{\xi}$-a.s. If $\xi = \delta_{z}$, $z \in \Zset$, we write $\PP_{z}$ and $\E_{z}$ instead of $\PP_{\delta_{z}}$ and $\E_{\delta_{z}}$, respectively. For $x^{1},\ldots,x^{k}$ being the iterates of any stochastic first-order method, we denote $\mathcal{F}_{k} = \sigma(x^{j}, j \leq k)$ and write $\E_{k}$ as an alias for $\E[\cdot | \mathcal{F}_{k}]$.
\begin{lemma}[\textnormal{\textit{Cauchy Schwartz inequality}}]
% \label{lem:CS}
For any $a,b,x_1, \ldots, x_n \in \R^d$ and $c > 0$ the following inequalities hold:
\begin{equation}
  \label{eq:inner_prod}
  2\langle a,b \rangle  \leq \frac{\norm{a}^2}{c} + c
  \norm{b}^2,
\end{equation}
\begin{equation}
  \label{eq:inner_prod_and_sqr}
  \norm{a+b}^2 \leq \left(1 + \frac{1}{c}\right)\norm{a}^2 + (1+c)
  \norm{b}^2.
\end{equation}
% \begin{equation}
%   \label{eq:sum_sqr}
%   \left\|\sum\limits_{i=1}^n x_i\right\|^2 \leq n \cdot \sum\limits_{i=1}^n \|x_i\|^2.
% \end{equation}
\end{lemma}
\section{Problem and Assumptions}
\label{sec:PaA}

In this paper, we study the minimization problem
\begin{equation}
\label{eq:erm}
\min_{x \in \R^d} \Bigg[f(x) := \E_{Z \sim \pi} [F(x, Z)]\Bigg]
\,,
\end{equation}
    where access to the function $f$ and its gradient are available only through the noisy oracle $F(x, Z)$ and $\nabla F(x, Z)$ respectively. We start by presenting two classical regularity constraints on the target function \emph{f}:
% In the following presentation we impose at least one of the following regularity constraint on the underlying function $f$ itself:
\begin{assumption}
\label{as:lipsh_grad}
The function $f$ is $L$-smooth on $\R^d$ with $L > 0$, i.e. it is continuously differentiable and there exists a constant $L > 0$ such that the following inequality holds for all $x,y \in \R^d$:
\begin{equation*}
    \label{L-smooth1}
    \norm{\nabla f(x) - \nabla f(y) } \leq L \norm{x-y }.
\end{equation*}
\end{assumption}

\begin{assumption}
\label{as:strong_conv}
The function $f$ is $\mu$-strongly convex on $\R^d$, i.e. it is continuously differentiable, and there exists a constant $\mu > 0$ such that the following inequality holds for all $x,y \in \R^d$:
\begin{equation*}
 \frac{\mu}{2}\norm{x-y}^2 \leq f(x) - f(y) - \langle \nabla f(y), x-y \rangle\,.
\end{equation*}
\end{assumption}
Next we specialize our assumption on the sequence of noise variables $\{Z_i\}_{i=0}^{\infty}$. Assumption \ref{as:Markov_noise_UGE} is also considered to be classical in the case of stochastic optimization with the Markovian noise \cite{dorfman2022adapting,sun2018markov,doan2020convergence}. It allows us to deal with finite state-space Markov chains with irreducible and aperiodic transition matrix. 
%In order to avoid Yet our definition of the mixing time $\taumix$ is more classical in the probability literature \cite{paulin_spectral}, and is slightly different from the one considered e.g. in \cite{even2023stochastic,mao20}.
% We consider here the general setting of $\{Z_i\}_{i=0}^{\infty}$ being a time-homogeneous Markov chain. We denote by $\MKQ$ the corresponding Markov kernel and impose the following assumption on the mixing properties of $\MKQ$:

\begin{assumption}
\label{as:Markov_noise_UGE}
$\{Z_i\}_{i=0}^{\infty}$ is a stationary Markov chain on $(\Zset,\Zsigma)$ with Markov kernel $\MKQ$ and unique invariant distribution $\pi$. Moreover, $\MKQ$ is uniformly geometrically ergodic with mixing time $\taumix \in \nset$, i.e., for every $k \in \nset$,
\begin{equation*}
 \dobrush(\MKQ^k) = \sup_{z,z' \in \Zset} (1/2) \norm{\MKQ^k(z, \cdot) - \MKQ^k(z',\cdot)}_{\mathsf{TV}} \leq (1/4)^{\lfloor k / \taumix \rfloor}\,.
\end{equation*}
\end{assumption}
Now we specify our assumption on the stochastic gradient estimator. The majority of existing literature on stochastic first order methods for solving \eqref{eq:erm} utilizes \textit{strong  growth condition} \cite{schmidt2013fast} or \textit{uniformly bounded variance} \cite{lan2008} as they allow to prove the convergence quite straightforwardly. However, these assumptions narrow down the set of target functions that can be considered rather strongly and there are several kinds of relaxation of it \cite{vaswani2019fast}, where gradient differences are bounded by the norm of the true gradient and a certain bias. Instead of this, we propose to use the following assumption: 
\begin{assumption}
\label{as:bounded_markov_noise_UGE}
For all $x \in \rset^{d}$ it holds that $\E_{\pi}[\nabla F(x, Z)] = \nabla f(x)$. Moreover, for all $z \in \Zset$ and $x \in \rset^{d}$ it holds that
\begin{equation}
\label{eq:growth_condition}
\norm{\nabla F(x, z) - \nabla f(x)}^{2} \leq \cmax^{2} + \dmax^{2} \| x-x^* \|^{2}\,.
\end{equation}
\end{assumption}
% One can notice, that this consideration is classical for solving variational inequalities \cite{hsieh2020explore,doi:10.1137/15M1031953,gorbunov2022stochastic}, since bounding by the norm of the operator does not allow to prove convergence in such problem statement. 
It is one way or another much weaker, then strong grows condition and uniformly bounded variance and, to the best of our knowledge, seem to be new for analyzing accelerated methods for solving stochastic optimization problems. One can notice that unlike the i.i.d. case, we are forced to require the almost sure bound in \eqref{eq:growth_condition} rather than in expectation. This issue inevitably arises when dealing with Markovian stochasticity due to the impossibility of using the expectation trick \cite{beznosikov2020distributed}, and has not yet been solved by any authors dealing with such type of stochasticity \cite{dorfman2022adapting,sun2018markov,doan23}. Either way, there are advantages to this approach as well. If we additionally require our noisy oracle $F(x, Z)$ to be $\Tilde{L}-$Lipschiz, then Assumption \ref{as:bounded_markov_noise_UGE} is automatically satisfied. Formally, if for any $x,y \in \R^d$, $$\norm{\nabla F(x, z) - \nabla F(y, z)} \leq \Tilde{L}(z)\norm{x - y},$$ for $\Tilde{L}:\Zset\rightarrow\rset^+$ with $\sup|\Tilde{L}|<\infty$, then 
\begin{equation*}
    \begin{split}
        \norm{\nabla F(x, z) - \nabla f(x)}^2 &\leq 3\norm{\nabla F(x, z) - \nabla F(x^*\!, z)}^2 + 3\norm{\nabla F(x^*\!, z) - \nabla f(x^*)}^2
        \\&
        + 3\norm{\nabla f(x) - \nabla f(x^*)}^2
        \\&\leq3(\norm{\Tilde{L}}^2+L^2)\norm{x-x^*}^2 + 3\norm{\nabla F(x^*\!, z) - \nabla f(x^*)}^2,
    \end{split}
\end{equation*}
taking $\sigma = \sqrt{3\norm{\nabla F(x^*\!, z) - \nabla f(x^*)}}$ and $\delta = \sqrt{6\max(L, \norm{\Tilde{L}})}$ gives Assumption \ref{as:bounded_markov_noise_UGE}.
% The assumption \ref{as:bounded_markov_noise_UGE} resembles the strong growth condition \cite{vaswani2019fast}, which is classical for the overparametrized learning setup \cite{vaswani2019fast, DBLP:journals/corr/abs-1905-09997}. The main difference is that \ref{as:bounded_markov_noise_UGE} concerns the almost sure bound in \eqref{eq:growth_condition}, which is unavoidable when dealing with uniformly geometrically ergodic Markovian noise \ref{as:Markov_noise_UGE}. With the assumptions \ref{as:Markov_noise_UGE} and \ref{as:bounded_markov_noise_UGE} we can proof the result on the mean squred error of the stochastic gradient estimate computed over batch size $n$ under arbitrary initial distribution.
\newpage
\section{Main results}
We start by introducing our version of Nesterov accelerated SGD.
It utilizes the idea from \cite{beznosikov2023first} of using exactly the number of samples that comes from the truncated geometric distribution with truncation parameter to be specified later (see Theorem \ref{th:acc}) in order to obtain optimal computational complexity of the algorithm.
% The main feature of Algorithm \ref{alg:AGD_ASGD} is that the number of samples used during the $k$-th gradient computation scales as $2^{J_k}$, where $J_k$ comes from a truncated geometric distribution. The truncation parameter needs to be adopted (see \ref{th:acc}) in order to control the computational complexity of the algorithm. 
{\tiny \begin{algorithm}[h!]
   \caption{\texttt{Markov Accelerated GD}}
   \label{alg:AGD_ASGD}
\begin{algorithmic}[1]
\State {\bf Parameters:} stepsize $\gamma>0$, momentums $\theta, \eta$, number of iterations $\nbiter$, batchsize limit $\batchbound$
\State {\bf Initialization:} choose  $x^0  = x^0_f$
\For{$k = 0, 1, 2, \dots, \nbiter-1$}
    \State $x^k_g = \theta x^k_f + (1 - \theta) x^k$ \label{line_acc_1}
    \State Sample $\textstyle{J_k \sim \text{Geom}\left(1/2\right)}$
    \State
    \text{\small{ 
    $g^{k} = g^{k}_0 +
        \begin{cases}
        \textstyle{2^{J_k}\hspace{-0.1cm}\left( g^{k}_{J_k}  - g^{k}_{J_k - 1} \right)}, &\hspace{-0.25cm} \text{if } 2^{J_k} \leq \batchbound \\
        0, & \hspace{-0.25cm}\text{otherwise}
        \end{cases}
    $with
    $
    \textstyle{g^k_j = \frac{1}{2^j} \sum\nolimits_{i=1}^{2^j} \nabla F
    (x^{k}_g, Z_{T^{k} + i})}
    $
    }} \label{line_acc_g}
    \State    $\textstyle{x^{k+1}_f = x^k_g - \gamma g^k}$ \label{line_acc_2}
    \State    $\textstyle{x^{k+1} = \eta x^{k+1}_f + (1 - \eta)x^k_f}$ \label{line_acc_3}
    \State    $\textstyle{T^{k+1} = T^{k} + 2^{J_{k}}}$\label{line_counter}
\EndFor
\end{algorithmic}
\end{algorithm}}
\\The key idea behind randomized batch size is to reduce the bias of the stochastic gradient estimator. Motivation for this is irrefutably natural as under the Markovian stochastic gradients oracles this bias appears by itself. Indeed, one can easily show the fact that:
\[
\E_{k}[\nabla F(x^{k},Z_{T^{k}+i})] \neq \nabla f(x^{k})\,.
\]
% \\Randomized batch size allows for efficient \emph{bias} reduction in the stochastic gradient estimates and can be seen as a particular case of the so called multilevel MCMC \cite{glynn2014exact, giles08}. In the optimization context this approach was successfully used by \cite{dorfman2022adapting} for the non-convex problems. Indeed, this bias naturally appears under the Markovian stochastic gradients oracles. It is easy to see that, with the counter $T^{k}$ defined in \ref{line_counter}, we have
% \[
% \E_{k}[\nabla F(x^{k},Z_{T^{k}+i})] \neq \nabla f(x^{k})\,.
% \]
In a subsequent part, we show how the bias of the gradient estimator introduced in line \ref{line_acc_g} of Algorithm \ref{alg:AGD_ASGD} scales with the truncation parameter $M$. To obtain proper dependence, we first need to introduce auxiliary Lemma \ref{lem:tech_markov}, which is to constrain the gradient estimator with a simpler structure. In particular, we bound MSE for sample average approximation computed over batch size $n$ under arbitrary initial distribution. We emphasise that it is extremely essential to have the bound for MSE under arbitrary initial distribution $\xi$, because in the proof of our Theorem \ref{th:acc} we will unavoidably manage the conditional expectations w.r.t. the previous iterate.
\begin{lemma}\label{lem:tech_markov}
Consider Assumptions \ref{as:Markov_noise_UGE} and \ref{as:bounded_markov_noise_UGE}. Then, for any $n \geq 1$ and $x \in \rset^{d}$, it holds that
\begin{equation}
\label{eq:var_bound_stationary_app}
{
\E_{\pi}\Big[\norm{\frac{1}{n}\sum\nolimits_{i=1}^{n}\nabla F(x, Z_{i}) - \nabla f(x)}^2\Big]
    \leq
    \frac{8 \taumix}{n} \left( \cmax^2 + \dmax^2 \| x - x^* \|^2 \right)
}\,.
\end{equation}
Moreover, for any initial distribution $\xi$ on $(\Zset,\Zsigma)$, that
\begin{equation}
\label{eq:var_bound_any_app}
{
\E_{\xi}\Big[\norm{\frac{1}{n}\sum\nolimits_{i=1}^{n}\nabla F(x, Z_i) - \nabla f(x)}^2\Big] \leq \frac{C_{1} \taumix}{n} \left( \cmax^2 + \dmax^2 \| x-x^* \|^2 \right)
}\,,
\end{equation}
where $C_{1} = 16(1 + \frac{1}{\ln^{2}{4}})$.
\end{lemma}
\begin{proof}
    By \cite[Lemma~19.3.6 and Theorem~19.3.9]{douc:moulines:priouret:soulier:2018}, for any two probabilities $\xi,\xi'$ on $(\Zset,\Zsigma)$ there is a \emph{maximal exact coupling} $(\Omega,\mathcal{F},\PPcoupling{\xi}{\xi'},Z,Z',T)$ of $\PP ^{\MKQ}_{\xi}$ and $\PP ^{\MKQ}_{\xi'}$, that is,
\begin{equation}
\label{eq:coupling_time_def_markov}
\tvnorm{\xi \MKQ^n- \xi'\MKQ^n} = 2 \PPcoupling{\xi}{\xi'}(T > n)\,.
\end{equation}
We write $\PEcoupling{\xi}{\xi'}$ for the expectation with respect to $\PPcoupling{\xi}{\xi'}$. Using the coupling construction \eqref{eq:coupling_time_def_markov},
\begin{multline*}
{\E^{1/2}_{\xi}\Big[\norm{\sum_{i=1}^{n}\{\nabla F(x, Z_i) - \nabla f(x)\}}^2\Big]
\leq \E^{1/2}_{\pi}\Big[\norm{\sum_{i=0}^{n-1}\nabla F(x, Z_i) - \nabla f(x)}^{2}\Big]} ~+ \\
{\PEcoupling{\xi}{\pi}^{1/2}\Big[\norm{\sum_{i=0}^{n-1}\{\nabla F(x, Z_i) - \nabla F(x, Z^{\prime}_i)\}}^{2}\Big]}\,.
\end{multline*}
The first term is bounded with \eqref{eq:var_bound_stationary_app}. Moreover, with \eqref{eq:coupling_time_def_markov} and Assumption \ref{as:bounded_markov_noise_UGE}, we get
\begin{align*}
    \norm{\sum_{i=0}^{n-1}\{\nabla F(x, Z_i) - \nabla F(x, Z^{\prime}_i)\}}^{2} 
    &\leq
    8 \left( \cmax^2 + \dmax^2 \| x - x^* \|^2 \right) \Bigg(\sum_{i=0}^{n-1}\indiacc{Z_i \neq Z^{'}_{i}}\Bigg)^{2}
    \\
    &= 8 \left( \cmax^2 + \dmax^2 \| x - x^* \|^2 \right) \Bigg(\sum_{i=0}^{n-1}\indiacc{T > i}\Bigg)^{2} 
    \\
    &\leq
    16 \left( \cmax^2 + \dmax^2 \| x - x^* \|^2 \right) \sum_{i=1}^{\infty} i\, \indiacc{T > i}\,.
\end{align*}
Thus, using the Assumption \ref{as:Markov_noise_UGE}, we bound
$$\PEcoupling{\xi}{\pi}\Big[\sum_{i=1}^{\infty} i\, \indiacc{T > i}\Big] = \sum_{i=1}^{\infty} i \PPcoupling{\xi}{\xi'}(T > i) = \sum_{i=1}^{\infty} i (1/4)^{\lfloor i / \taumix \rfloor} \leq 4 \sum_{i=1}^{\infty} i (1/4)^{ i / \taumix }\,.$$
Now we set $\rho = (1/4)^{1/ \taumix}$ and use an upper bound
\begin{equation*}
    \begin{split}
        \sum_{k=1}^{\infty}k \rho^{k}
        \leq \rho^{-1}\int\limits_{0}^{+\infty}x^{p}\rho^{x}\,d x &\leq \rho^{-1}\left(\ln{\rho^{-1}}\right)^{-2} \Gamma(2) 
        \\&= 
        \rho^{-1}\left(\ln{\rho^{-1}}\right)^{-2} = \frac{\taumix^{2}}{(1/4)^{1/\taumix} \ln^{2}{4}}\,.
    \end{split}
\end{equation*}
Combining the bounds above yields
\[
\E_{\xi}\Big[\norm{\frac{1}{n}\sum\limits_{i=1}^{n}\nabla F(x, Z_i) - \nabla f(x)}^2\Big] \leq \Big(\frac{c_{1} \taumix}{n} + \frac{c_{2} \taumix^{2}}{n^2}\Big) \left( \cmax^2 + \dmax^2 \| x - x^* \|^2 \right)\,,
\]
where $c_{1} = 16$, $c_{2} = \frac{128 (1/4)^{-1/\taumix}}{\ln^{2}{4}}$. Now we consider the two cases. If $n < c_{1}\taumix$, we get from Minkowski's inequality that
\[
\E_{\xi}\Big[\norm{\frac{1}{n}\sum\limits_{i=1}^{n}\nabla F(x, Z_i) - \nabla f(x)}^2\Big] \leq 2\cmax^2 + 2\dmax^2 \| x - x^* \|^2\,,
\]
and \eqref{eq:var_bound_any_app} holds. If $n > c_{1}\taumix$, it holds that
\[
\frac{c_{2} \taumix^{2}}{n^2} \left( \cmax^2 + \dmax^2 \| \nabla f(x) \|^2 \right) \leq \frac{c_{2} \taumix^{2}}{n c_{1} \taumix}\left( \cmax^2 + \dmax^2 \| x - x^* \|^2 \right)\,,
\]
and we gain \eqref{eq:var_bound_any_app} too.\hfill$\square$
\end{proof}
We are now ready to bound the MSE for the gradient estimator introduced in line \ref{line_acc_g} of Algorithm \ref{alg:AGD_ASGD}. From Lemma \ref{lem:expect_bound_grad_appendix}, we obtain a desired linear dependence of the error reduction on the parameter $M$. 
\begin{lemma}
\label{lem:expect_bound_grad_appendix}
Consider Assumptions \ref{as:Markov_noise_UGE} and \ref{as:bounded_markov_noise_UGE}. Then for the gradient estimates $g^k$ from line \ref{line_acc_g} Algorithm \ref{alg:AGD_ASGD} it holds that $\E_k[g^k] = \E_k[g^{k}_{\lfloor \log_2 \batchbound \rfloor}]$. Moreover, 
\begin{align}
\label{eq:var_bounds_random_grad}
&\E_k[\| \nabla f(x^k_g) - g^k\|^2] \leq 13C_{1} \taumix \log_2 M (\cmax^2 + \dmax^2 \| x^k_g - x^*\|^2)\,, \\
&\| \nabla f(x^k_g) - \E_{k}[g^k]\|^2 \leq 2C_{1}\taumix \batchbound^{-1} (\cmax^2 + \dmax^2 \| x^k_g - x^*\|^2)\,, \nonumber 
\end{align}
where $C_1$ is defined in \eqref{eq:var_bound_any_app}.
\end{lemma}
\begin{proof}
To show that $\E_k[g^k] = \E_k[g^{k}_{\lfloor \log_2 \batchbound \rfloor}]$ we simply compute conditional expectation w.r.t. $J_k$:
\begin{equation}
\label{eq:tech:lem3}
    \begin{split}
        \E_k[g^k] &= \E_k\left[\E_{J_k}[g^k]\right] =
        \E_k[g^k_0] + \sum\limits_{i=1}^{\lfloor \log_2 \batchbound \rfloor} \Prob\{J_k = i\} \cdot 2^i \E_k[g^{k}_{i}  - g^{k}_{i-1}] \\
        &= \E_k[g^k_0] + \sum\limits_{i=1}^{\lfloor \log_2 \batchbound \rfloor} \E_k[g^{k}_{i}  - g^{k}_{i-1}] = \E_k[g^{k}_{\lfloor \log_2 \batchbound \rfloor}]\,.
    \end{split}
\end{equation}
We start with the proof of the first statement of \eqref{eq:var_bounds_random_grad} by taking the conditional expectation for $J_k$:
\begin{align*}
    &\E_{k}[\| \nabla f(x^k_g) - g^k\|^2]
    \leq
    2\E_{k}[\| \nabla f(x^k_g) - g^k_0\|^2] + 2\E_{k}[\| g^k - g^k_0\|^2]
    \\
    &\quad=
    2\E_{k}[\| \nabla f(x^k_g) - g^k_0\|^2] + 2 \sum\nolimits_{i=1}^{\lfloor \log_2 \batchbound \rfloor} \Prob\{J_k = i\} \cdot 4^i \E_k[\|g^{k}_{i}  - g^{k}_{i-1}\|^2] \\
    &\quad=
    2\E_{k}[\| \nabla f(x^k_g) - g^k_0\|^2] + 2\sum\nolimits_{i=1}^{\lfloor \log_2 \batchbound \rfloor} 2^i \E_k[\|g^{k}_{i}  - g^{k}_{i-1}\|^2] \\
    &\quad\leq
    2\E_{k}[\| \nabla f(x^k_g) - g^k_0\|^2]\,\, + \\
    &\quad\quad+4\sum\nolimits_{i=1}^{\lfloor \log_2 \batchbound \rfloor} 2^i \left(\E_k[\|\nabla f(x^k_g)  - g^{k}_{i-1}\|^2 + \E_k[\|g^{k}_{i}  - \nabla f(x^k_g)\|^2] \right)\,.
\end{align*}
To bound $\E_{k}[\| \nabla f(x^k_g) - g^k_0\|^2]$, $\E_{k}[\|\nabla f(x^k_g)  - g^{k}_{i-1}\|^2$, $\E_{k}[\|g^{k}_{i}  - \nabla f(x^k_g)\|^2]$, we apply Lemma \ref{lem:tech_markov} and get
\begin{align*}
    \E_{k}[\| &\nabla f(x^k_g) - g^k\|^2]
     \\&\leq2 \cmax^{2} + 2\dmax^2 \| x_g^k - x^*\|^2 + 12\sum\nolimits_{i=1}^{\lfloor \log_2 \batchbound \rfloor} 2^i \cdot \frac{C_{1} \taumix}{2^{i}} (\cmax^2 + \dmax^2 \| x_g^k - x^*\|^2) \\
    &\leq
    13C_{1} \taumix \log_2 M(\cmax^2 + \dmax^2 \| x_g^k - x^*\|^2)\,.
\end{align*}
To show the second part of the statement, we use \eqref{eq:tech:lem3} and get
\[
\| \nabla f(x^k_g) - \E_{k}[g^k]\|^2 = \| \nabla f(x^k) - \E_k[g^{k}_{\lfloor \log_2 \batchbound \rfloor}]\|^2\,.
\]
Using Lemma \ref{lem:tech_markov} and $2^{\lfloor \log_{2}\batchbound \rfloor} \geq \batchbound/2$ finishes the proof.\hfill$\square$
\end{proof}
We also note that our proofs of Lemma \ref{lem:tech_markov} and Lemma \ref{lem:expect_bound_grad_appendix} rely on the proofs of Lemmas 1 and 2 of \cite{beznosikov2023first}, but for the sake of clarity of the narrative we give them in full.\\
Now, before we move on to the proof of our major result, we first need to introduce two descent lemmas: 
\begin{lemma}
\label{lem:tech_lemma_acc_gd}
Consider Assumptions \ref{as:lipsh_grad} and \ref{as:strong_conv} be satisfied. Then for the iterates of Algorithm \ref{alg:AGD_ASGD} with $\theta = (1 - \eta) / (\beta - \eta)$, $\theta > 0$, $\eta \geq 1$, it holds that
\begin{align}
\label{eq:acc_temp4}
    \E_k[\|x^{k+1} - x^*\|^2]
     \leq&
    (1 + \alpha \gamma \eta)( 1 - \beta) \| x^k - x^*\|^2 + (1 + \alpha \gamma \eta) \beta\|x^k_g - x^*\|^2 
    \notag\\
    &
    + (1 + \alpha \gamma \eta) (\beta^2 - \beta )\|x^k - x^k_g\|^2
    + \eta^2 \gamma^2 \EE_{k}[\| g^k \|^2] 
    \notag\\
    &
    - 2 \eta \gamma \langle \nabla f(x^k_g), \eta x^k_g + (1 - \eta)  x^k_f  - x^*\rangle 
    \notag\\
    &
    + \frac{\eta \gamma}{\alpha} \|\EE_k[g^k] - \nabla f(x^k_g)\|^2\,,
\end{align}
where $\alpha > 0$ is any positive constant.
\end{lemma}
\begin{proof}
We start with lines \ref{line_acc_3} and \ref{line_acc_2} of Algorithm \ref{alg:AGD_ASGD}:
\begin{align*}
    \|x^{k+1} - &x^*\|^2
    =
    \| \eta x^{k+1}_f + (1 - \eta)x^k_f - x^*\|^2 = 
    \| \eta x^k_g - \eta\gamma g^k + (1 - \eta)x^k_f - x^*\|^2
    \notag\\
    =&
    \| \eta x^k_g + (1 - \eta)x^k_f - x^*\|^2 + \gamma^2 \eta^2 \| g^k \|^2 
    - 2 \gamma \eta \langle  g^k, \eta x^k_g + (1 - \eta)x^k_f - x^*\rangle.
\end{align*}
Using straightforward algebra, we get
\begin{align*}
    \|x^{k+1} - x^*\|^2
    =&
    \| \eta x^k_g + (1 - \eta)x^k_f - x^*\|^2 
    - 2 \gamma \eta \langle \nabla f(x^k_g), \eta x^k_g + (1 - \eta)x^k_f - x^* \rangle
   \notag\\
    &
    - 2 \gamma \eta \langle \EE_k[g^k] - \nabla f(x^k_g), \eta x^k_g + (1 - \eta)x^k_f - x^* \rangle + \gamma^2 \eta^2 \| g^k \|^2
    \notag\\
    &
    - 2 \gamma \eta \langle g^k - \EE_k[g^k], \eta x^k_g + (1 - \eta)x^k_f - x^* \rangle
    \notag\\
    \leq&
    (1 + \alpha \eta \gamma)\| \eta x^k_g + (1 - \eta)x^k_f - x^* \|^2 + \frac{\gamma \eta}{\alpha} \|\EE_k[g^k] - \nabla f(x^k_g)\|^2.
    \notag\\
    &
    - 2 \gamma \eta \langle \nabla f(x^k_g), \eta x^k_g + (1 - \eta)x^k_f - x^* \rangle + \gamma^2 \eta^2 \| g^k \|^2
    \notag\\
    &
    - 2 \gamma \eta \langle g^k - \EE_k[g^k], \eta x^k_g + (1 - \eta)x^k_f - x^* \rangle
\end{align*}
In the last step we also applied Cauchy-Schwartz inequality in the form \eqref{eq:inner_prod} with $c > 0$. Taking the conditional expectation, we get
\begin{align}
\label{eq:acc_temp1}
\E_{k}[\|x^{k+1} - x^*\|^2] \leq& (1 + \alpha \eta \gamma) \|\eta x^k_g + (1 - \eta)x^k_f - x^*\|^2 
\notag\\
    &
    - 2 \gamma \eta \langle \nabla f(x^k_g), \eta x^k_g + (1 - \eta)x^k_f - x^* \rangle
    \notag\\
    &
    + \gamma^2 \eta^2 \E_{k}[\| g^k \|^2]  + \frac{\gamma \eta}{\alpha} \|\EE_k[g^k] - \nabla f(x^k_g)\|^2\,.
\end{align}
Now let us handle expression $\| \eta x^k_g + (1 - \eta)x^k_f  - x^*\|^2$ for a while. Taking into account line \ref{line_acc_1} and the choice of $\theta$ such that $\theta = (1 - \eta) / (\beta - \eta)$ (in particular, $\beta = \eta + (1-\eta)/\theta$ and $(1-\eta)(\theta - 1)/\theta = 1 - \beta$), we get
\begin{align*}
\eta x^k_g + (1 - \eta)  x^k_f &= \eta x^k_g + \frac{(1 - \eta)}{\theta}x^k_g - \frac{(1 - \eta)(1 - \theta)}{\theta} x^k = \beta x^k_g + (1 - \beta) x^k
% \notag\\
% &=
% (\eta + (1 - p) \beta) x^k_g + \eta (p \eta^{-1} - 1)x^k_f + (1- p)(1 - \beta) x^k
% \notag\\
% &=
% (\eta + (1 - p) \beta) x^k_g + \eta (\beta p \eta^{-1} - 1)\theta x^k_f + (1- p)(1 - \beta) x^k
% \notag\\
% &=
% (\eta + (1 - p) \beta) x^k_g + \eta (\beta p \eta^{-1} - 1)(x^k_g - (1 - \theta) x^k) + (1- p)(1 - \beta) x^k
% \notag\\
% &=
% (\eta + (1 - p) \beta) x^k_g + \eta (\beta p \eta^{-1} - 1)(x^k_g - (1 - \theta) x^k) + (1- p)(1 - \beta) x^k
% \notag\\
% &=
% \beta x^k_g - \eta (\beta p \eta^{-1} - 1)(1 - \theta) x^k + (1- p)(1 - \beta) x^k
% \notag\\
% &=
% \beta x^k_g + p(1 - \beta) x^k + (1- p)(1 - \beta) x^k
% \notag\\
% &=
% \beta x^k_g + (1 - \beta) x^k
% \,.
\end{align*}
Substituting into $\|\eta x^k_g + (1 - \eta)x^k_f - x^*\|^2$, we get
\begin{align}
\label{eq:acc_temp3}
\| \eta &x^k_g + (1 - \eta) x^k_f - x^*\|^2 = \| \beta x^k_g + (1 - \beta) x^k - x^*\|^2
\notag\\
&=
\| x^k - x^* + \beta (x^k_g - x^k)\|^2
\notag\\
&=
\| x^k - x^*\|^2 + 2 \beta \langle x^k - x^*, x^k_g - x^k \rangle  + \beta^2\|x^k - x^k_g\|^2
\notag\\
&=
\| x^k - x^*\|^2 + \beta \left( \|x^k_g - x^*\|^2 - \| x^k - x^*\|^2 - \| x^k_g - x^k\|^2\right)  + \beta^2\|x^k - x^k_g\|^2
\notag\\
&=
( 1 - \beta) \| x^k - x^*\|^2 + \beta\|x^k_g - x^*\|^2  + (\beta^2 - \beta )\|x^k - x^k_g\|^2.
\end{align}
% Again with line \ref{line_acc_1} and the choice of $\theta$ such that $\theta = (p \eta^{-1} - 1) / (\beta p \eta^{-1} - 1)$ (in particular, $\eta^{-1} p (1 - \beta) = (1 - \beta p \eta^{-1}) (1 - \theta)$ and $(\beta p \eta^{-1} - 1)\theta = (p \eta^{-1} - 1)$), one can also note
% \begin{align}
% \label{eq:acc_temp8}
% \eta x^k_g &+ (p - \eta)x^k_f + (1- p)(1 - \beta) x^k +(1 - p) \beta x^k_g - x^*
% \notag\\
% &= (\eta + (1 - p) \beta)x^k_g + (p - \eta)x^k_f + (1- p)(1 - \beta) x^k - x^*
% \notag\\
% &= \eta p^{-1}\left( (p + (1 - p) p^{-1}\eta \beta)x^k_g + (p \eta^{-1} - 1) p x^k_f + (1- p)(1 - \beta) p \eta^{-1} x^k - \eta^{-1} p x^* \right)
% \notag\\
% &= \eta p^{-1}\left( (p + (1 - p) p^{-1}\eta \beta)x^k_g + (p \eta^{-1} - 1) p x^k_f + (1- p)(1 - \beta p \eta^{-1}) (1 - \theta) x^k - \eta^{-1} p x^* \right)
% \notag\\
% &= \eta p^{-1}\left( (p + (1 - p) p^{-1}\eta \beta)x^k_g + (p \eta^{-1} - 1) p x^k_f + (1- p)(1 - \beta p \eta^{-1}) (x^k_g - \theta x^k_f) - \eta^{-1} p x^* \right)
% \notag\\
% &= \eta p^{-1}\left( x^k_g + (p \eta^{-1} - 1) p x^k_f - (1- p)(1 - \beta p \eta^{-1}) \theta x^k_f - \eta^{-1} p x^* \right)
% \notag\\
% &= \eta p^{-1}\left( x^k_g + (p \eta^{-1} - 1) p x^k_f - (1- p)(p \eta^{-1} - 1) x^k_f - \eta^{-1} p x^* \right)
% \notag\\
% &= \eta p^{-1}\left( x^k_g + (p \eta^{-1} - 1)  x^k_f - \eta^{-1} p x^* \right)
% .
% \end{align}
Combining  \eqref{eq:acc_temp3} with \eqref{eq:acc_temp1}, we finish the proof. \hfill$\square$
\end{proof}

\begin{lemma}
\label{lem:tech_lemma_acc_gd_2}
Let Assumptions \ref{as:lipsh_grad} and \ref{as:strong_conv} be satisfied. Let problem \eqref{eq:erm}  be solved by Algorithm \ref{alg:AGD_ASGD}. Then for any $u \in \R^d$, we get
\begin{align*}
        \EE_k[f(x^{k+1}_f)]
        \leq&
        f(u) - \langle \nabla f(x^k_g), u - x^k_g \rangle - \frac{\mu}{2} \| u - x^k_g\|^2  - \frac{\gamma}{2} \|\nabla f(x^k_g)\|^2 
        \\
        &
        + \frac{\gamma}{2} \|\EE_k[g^k] - \nabla f(x^k_g) \|^2 + \frac{L \gamma^2 }{2}\EE_k[\| g^k\|^2].
\end{align*}
\end{lemma}
\begin{proof}
Using  Assumption \ref{as:lipsh_grad} and line \ref{line_acc_2} of Algorithm \ref{alg:AGD_ASGD}, we get
\begin{align*}
f(x^{k+1}_f) &\leq f(x^k_g) + \langle \nabla f(x^k_g), x^{k+1}_f - x^k_g \rangle + \frac{L}{2}\| x^{k+1}_f - x^k_g\|^2 \\
        &= f(x^k_g) -  \gamma \langle \nabla f(x^k_g), g^k \rangle + \frac{L  \gamma^2 }{2}\| g^k\|^2 \\
        &=
        f(x^k_g) -  \gamma  \langle \nabla f(x^k_g), \nabla f(x^k_g) \rangle - \gamma \langle \nabla f(x^k_g), \EE_k[g^k] - \nabla f(x^k_g) \rangle 
        \\
        &\quad- \gamma \langle \nabla f(x^k_g), g^k - \EE_k[g^k] \rangle + \frac{L\gamma^2}{2}\| g^k\|^2
        \\
        &\leq
        f(x^k_g) - \gamma \|\nabla f(x^k_g)\|^2 + \frac{\gamma}{2} \| \nabla f(x^k_g) \|^2 + \frac{\gamma}{2} \|\EE_k[g^k] - \nabla f(x^k_g) \|^2
        \\
        &\quad- \gamma \langle \nabla f(x^k_g), g^k - \EE_k[g^k] \rangle + \frac{L \gamma^2 }{2}\| g^k\|^2.
\end{align*}
Here we also used Cauchy-Schwartz inequality \eqref{eq:inner_prod} with $a =  \nabla f(x^k_g)$, $b =  \nabla f(x^k_g) - \EE_k[g^k]$ and $c = 1$. Taking the conditional expectation, we get
\begin{align*}
\EE_k[f(x^{k+1}_f)] \leq& f(x^k_g) - \frac{\gamma}{2} \|\nabla f(x^k_g)\|^2 + \frac{\gamma}{2} \|\EE_k[g^k] - \nabla f(x^k_g) \|^2 + \frac{L\gamma^2 }{2}\EE_k[\| g^k\|^2].
\end{align*}
Using Assumption \ref{as:strong_conv}  with $x= u$ and $y= x^k_g$, one can conclude that for any $u \in \R^d$ it holds
    \begin{align*}
        \EE_k[f(x^{k+1}_f)]
        \leq&
        ~f(u) - \langle \nabla f(x^k_g), u - x^k_g \rangle - \frac{\mu}{2} \| u - x^k_g\|^2  - \frac{\gamma}{2} \|\nabla f(x^k_g)\|^2 
        \\
        &+ \frac{\gamma}{2} \|\EE_k[g^k] - \nabla f(x^k_g) \|^2 + \frac{L \gamma^2 }{2}\EE_k[\| g^k\|^2].~~~~~~~~~~~~~~~~~~~~~~~~~~~~\square
    \end{align*}
\end{proof}
Taking into account all of the considerations above, we
can prove the following result:
\begin{theorem}\label{th:acc}
Consider Assumptions \ref{as:lipsh_grad} -- \ref{as:bounded_markov_noise_UGE}. Let the problem \eqref{eq:erm}
be solved by Algorithm \ref{alg:AGD_ASGD}. Then for $\beta, \theta, \eta, \gamma, \batchbound$ satisfying 
\begin{align*}
M = (1 + 2/&\beta),\quad
\beta = \sqrt{\frac{4\mu\gamma}{9}}, \quad\eta = \frac{9\beta}{2\mu \gamma} = \sqrt{\frac{9}{\mu \gamma}},
\\
\gamma \lesssim &\min \Bigg\{\frac{\mu^3}{\delta^4\tau^2}; \frac{1}{L}\Bigg\}, \quad  \theta = \frac{1 - \eta}{\beta - \eta},
\end{align*}
it holds that
\begin{align*}
    \EE\Bigg[&\|x^{N} - x^*\|^2 + \frac{18}{\mu} (f(x^{N}_f) - f(x^*)) \Bigg]
    \\
    &\lesssim 
    \exp\left( - N\sqrt{\frac{\mu \gamma}{9}}\right) \left[\| x^0 - x^*\|^2 + \frac{18}{\mu} (f(x^0) - f(x^*)) \right]
    + \frac{\sqrt{\gamma}}{\mu^{3/2}} C_{1} \taumix \log_2 M \cmax^2.
\end{align*}
\end{theorem}
\begin{proof}%[Proof of \ref{th:acc}]
Using Lemma \ref{lem:tech_lemma_acc_gd_2} with $u = x^*$ and $u = x^k_f$, we get
\begin{align*}
\EE_k[f(x^{k+1}_f)]
\leq&
f(x^*) - \langle \nabla f(x^k_g), x^* - x^k_g \rangle - \frac{\mu}{2} \|x^* - x^k_g\|^2  - \frac{\gamma}{2} \|\nabla f(x^k_g)\|^2 \\
&+ \frac{\gamma}{2} \|\EE_k[g^k] - \nabla f(x^k_g) \|^2 + \frac{L\gamma^2 }{2}\EE_k[\| g^k\|^2],
\end{align*}
\begin{align*}
        \EE_k[f(x^{k+1}_f)]
        \leq&
        f(x^k_f) - \langle \nabla f(x^k_g), x^k_f - x^k_g \rangle - \frac{\mu}{2} \| x^k_f - x^k_g\|^2  - \frac{\gamma}{2} \|\nabla f(x^k_g)\|^2 
        \\
        &+ \frac{\gamma}{2} \|\EE_k[g^k] - \nabla f(x^k_g) \|^2 + \frac{L \gamma^2 }{2}\EE_k[\| g^k\|^2].
\end{align*}
Summing the first inequality with coefficient $2\gamma \eta  $, the second with coefficient $2 \gamma \eta (\eta - 1) $ and \eqref{eq:acc_temp4}, we obtain
\begin{align*}
    \E_k[\|&x^{k+1} - x^*\|^2 + 2 \gamma \eta^2 f(x^{k+1}_f)]
    \\
    \notag \leq&
    (1 + \alpha \gamma \eta)( 1 - \beta) \| x^k - x^*\|^2 + (1 + \alpha \gamma \eta) \beta\|x^k_g - x^*\|^2 
    \\
    &+ (1 + \alpha \gamma \eta) (\beta^2 - \beta )\|x^k - x^k_g\|^2 - 2 \eta \gamma \langle \nabla f(x^k_g), \eta x^k_g + (1 - \eta)  x^k_f - x^*\rangle 
    \\
    &+ \eta^2 \gamma^2 \EE_{k}[\| g^k \|^2] + \frac{\eta \gamma}{\alpha} \|\EE_k[g^k] - \nabla f(x^k_g)\|^2
    \\
    & + 2 \gamma \eta \Big (f(x^*) - \langle \nabla f(x^k_g), x^* - x^k_g \rangle - \frac{\mu}{2} \|x^* - x^k_g\|^2  - \frac{\gamma}{2} \|\nabla f(x^k_g)\|^2 
    \\
        &+ \frac{\gamma}{2} \|\EE_k[g^k] - \nabla f(x^k_g) \|^2 + \frac{L \gamma^2 }{2}\EE_k[\| g^k\|^2]\Big)
    \\
    & + 2 \gamma \eta (\eta - 1) \Big ( f(x^k_f) - \langle \nabla f(x^k_g), x^k_f - x^k_g \rangle - \frac{\mu}{2} \| x^k_f - x^k_g\|^2  - \frac{\gamma}{2} \|\nabla f(x^k_g)\|^2 
        \\
        &+ \frac{\gamma}{2} \|\EE_k[g^k] - \nabla f(x^k_g) \|^2 + \frac{L \gamma^2 }{2}\EE_k[\| g^k\|^2]  
    \Big)
    \\
    \notag =&
    (1 + \alpha \gamma \eta)( 1 - \beta) \| x^k - x^*\|^2 + 2 \gamma \eta \left( \eta - 1\right)(f(x^{k}_f) - 2 \gamma \eta f(x^*))
    \\
    &\notag
    + \left((1 + \alpha \gamma \eta) \beta - \gamma \eta \mu\right)\|x^k_g - x^*\|^2
    \\
    &\notag
    + (1 + \alpha \gamma \eta) (\beta^2 - \beta )\|x^k - x^k_g\|^2 
    - \gamma^2 \eta^2 \|\nabla f(x^k_g)\|^2 
    \\
    &\notag
    + \left( \frac{\eta \gamma}{\alpha} + \gamma^2 \eta^2 \right) \|\EE_k[g^k] - \nabla f(x^k_g) \|^2 + \left( \eta^2 \gamma^2 + \gamma^3 \eta^2 L \right) \EE_k[\| g^k\|^2]
    \\
    \notag \leq&
    (1 + \alpha \gamma \eta)( 1 - \beta) \| x^k - x^*\|^2 + 2 \gamma \eta \left( \eta - 1\right)(f(x^{k}_f) - 2 \gamma \eta f(x^*))
    \\
    &\notag
    + \left((1 + \alpha \gamma \eta) \beta - \gamma \eta \mu\right)\|x^k_g - x^*\|^2
    \\
    &\notag
    + (1 + \alpha \gamma \eta) (\beta^2 - \beta )\|x^k - x^k_g\|^2 
    - \gamma^2 \eta^2 \|\nabla f(x^k_g)\|^2 
    \\
    &\notag
    + \eta \gamma \left( \frac{1}{\alpha} + \gamma \eta \right) \|\EE_k[g^k] - \nabla f(x^k_g) \|^2 + 8 \eta^2 \gamma^2 \left(  1 +  \gamma L \right) \EE_k[\| g^k - \nabla f(x^k_g)\|^2] 
    \\
    &\notag
    + \frac{1}{2} \eta^2 \gamma^2 \left(  1 +  \gamma L \right) \EE_k[\| \nabla f(x^k_g) \|^2]\,.
\end{align*}
In the last step, we also used \eqref{eq:inner_prod_and_sqr} with $c = 4$. Since $\gamma \leq \tfrac{9}{16L}$, the choice of $\alpha = \frac{\beta}{2\eta \gamma}$, $\beta  = \sqrt{16\mu \gamma / 9}$ gives
\begin{align*}
    &\beta  = \sqrt{16\mu \gamma / 9} \leq \sqrt{\mu / L} \leq 1,\\
    &(1 + \alpha \eta \gamma) (1 - \beta) = \left(1 + \frac{\beta}{2}\right) \left( 1 - \beta\right) \leq \left( 1 - \frac{\beta}{2}\right), 
    % \\
    % &\left((1 + \alpha \eta \gamma) \beta - \mu \gamma \eta \right) = \left( \beta + \frac{\beta^2}{2} -  \mu \gamma \eta \right) \leq \left(\frac{3 \beta}{2} - \mu \gamma \eta \right) \leq 0,
\end{align*}
and, therefore,
\begin{align*}
\E_{k}\bigl[\|x^{k+1} - x^*\|^2 &+ 2\gamma \eta^2 f(x^{k+1}_f)\bigr] 
\\
\leq&
    (1 - \beta / 2) \| x^k - x^*\|^2 + 2 \gamma \eta \left( \eta - 1\right)(f(x^{k}_f) - 2 \gamma \eta f(x^*))
    \\
&\notag
    + \eta^2 \gamma^2 \left( 1 + 2 /\beta \right) \|\EE_k[g^k] - \nabla f(x^k_g) \|^2
     \\
&\notag
    + 8 \eta^2 \gamma^2 \left(  1 +  \gamma L \right) \EE_k[\| g^k - \nabla f(x^k_g)\|^2] 
    \\
&\notag 
    + \left((1 + \alpha \gamma \eta) \beta - \gamma \eta \mu\right)\|x^k_g - x^*\|^2
\end{align*}
Subtracting $2\gamma \eta^2 f(x^*)$ from both sides, we get
\begin{align*}
    \E_{k}\bigl[\|x^{k+1} - x^*\|^2 &+ 2\gamma \eta^2 (f(x^{k+1}_f) - f(x^*))\bigr]
    \\
    \leq&
    \left( 1 - \beta / 2\right) \| x^k - x^*\|^2 + \left( 1 - 1/\eta\right) \cdot 2\gamma \eta^2 (f(x^k_f) - f(x^*))
    \\
    &\notag
    + \eta^2 \gamma^2 \left( 1 + 2 /\beta \right) \|\EE_k[g^k] - \nabla f(x^k_g) \|^2 
     \\
    &\notag
    + 8 \eta^2 \gamma^2 \left(  1 +  \gamma L \right) \EE_k[\| g^k - \nabla f(x^k_g)\|^2] 
    \\
    &\notag 
    + \left((1 + \alpha \gamma \eta) \beta - \gamma \eta \mu\right)\|x^k_g - x^*\|^2
\end{align*}
Applying Lemma \ref{lem:expect_bound_grad_appendix} and $\gamma L \leq 1$, one can obtain
\begin{align*}
    \E_{k}\bigl[\|x^{k+1} - x^*\|^2 &+ 2\gamma \eta^2 (f(x^{k+1}_f) - f(x^*))\bigr]
    \\
    \leq&
    \left( 1 - \beta / 2\right) \| x^k - x^*\|^2 + \left( 1 - 1/\eta\right) \cdot 2\gamma \eta^2 (f(x^k_f) - f(x^*))
    \\
    &\notag
    + \eta^2 \gamma^2 \left( 1 + 2 /\beta \right) \cdot 2C_{1}\taumix \batchbound^{-1} (\cmax^2 + \dmax^2 \| x^k_g - x^*\|^2) 
    \\
    &\notag
    + 16 \eta^2 \gamma^2 \cdot 13C_{1} \taumix \log_2 M (\cmax^2 + \dmax^2 \| x^k_g - x^* \|^2) 
    \\
    &\notag 
    + \left((1 + \alpha \gamma \eta) \beta - \gamma \eta \mu\right)\|x^k_g - x^*\|^2
\end{align*}
With $M \geq (1 + 2/\beta)$, $\sqrt{\gamma} \leq \frac{\mu^{\frac{3}{2}}}{1872C_1\tau\delta^2\log_2 M}$, $\alpha = \frac{\beta}{2\gamma\eta}$, $\beta = \frac{2}{3}\sqrt{\mu\gamma}$ and $\eta = \sqrt{\frac{9}{\mu\gamma}},$ we have:
\begin{align*}
    \centering
    (1 + \alpha \gamma \eta) \beta &- \gamma \eta \mu +  \eta^2 \gamma^2 \delta^2 \big(\left( 1 + 2 /\beta \right) \cdot 2C_{1}\taumix \batchbound^{-1} + 208 C_{1} \taumix \log_2 M\big) \\
    &\leq (1 + \alpha \gamma \eta)\beta - 3\sqrt{\mu\gamma} + \frac{C_1\taumix\delta^2\gamma}{\mu}\big(18+1872\log_2 M\big) \\
    &\leq \sqrt{\mu\gamma} - 3\sqrt{\mu\gamma} + 2\sqrt{\mu\gamma} \leq 0,
\end{align*}
and then,
\begin{align*}
    \E_{k}\bigl[\|x^{k+1} &- x^*\|^2 + 2\gamma \eta^2 (f(x^{k+1}_f) - f(x^*))\bigr]
    \\
    \leq&
    \big( 1 - \beta / 2\big) \| x^k - x^*\|^2 + \left( 1 - 1/\eta\right) \cdot 2\gamma \eta^2 (f(x^k_f) - f(x^*))
    \\
    &\notag
    + \Big(\eta^2 \gamma^2 \left( 1 + 2 /\beta \right) \cdot 2C_{1}\taumix \batchbound^{-1} + 16 \eta^2 \gamma^2 \cdot 13C_{1} \taumix \log_2 M\Big) \cmax^2 
    \\
    \leq&
    \max\left\{ \left( 1 - \beta / 2\right), \left( 1 - 1/\eta\right)\right\} \left[\| x^k - x^*\|^2 + 2\gamma \eta^2 (f(x^k_f) - f(x^*)) \right]
    \\
    &\notag
    + \Big(\eta^2 \gamma^2  \left( 1 + 2 /\beta \right) \cdot 2C_{1}\taumix \batchbound^{-1} + 16 \eta^2 \gamma^2 \cdot 13C_{1} \taumix \log_2 M\Big) \cmax^2.
\end{align*}
Using that $\eta \gamma = 9\beta / (2\mu)$, $\beta/2 = 1 /\eta$ and $\gamma \leq L^{-1}$, we have 
\begin{align*}
    \E_{k}\bigl[\|x^{k+1} - x^*\|^2 &+ 2\gamma \eta^2 (f(x^{k+1}_f) - f(x^*))\bigr]
    \\
    &\leq
    \left( 1 - \beta / 2\right) \left[\| x^k - x^*\|^2 + 2\gamma \eta^2 (f(x^k_f) - f(x^*)) \right]
    \\
    &\notag
    + \frac{81}{4}\beta^2 \mu^{-2} \Big(\left( 1 + 2 /\beta \right) \cdot 2C_{1}\taumix \batchbound^{-1} + 208 C_{1} \taumix \log_2 M\Big) \cmax^2
\end{align*}
Finally, we perform the recursion and substitute $\beta = \sqrt{4\mu \gamma /9}$:
\begin{align*}
    \EE\bigl[\|x^{N} - x^*\|^2 &+ 2\gamma \eta^2 (f(x^{N}_f) - f(x^*)) \bigr]
    \\
    \leq&
    \left( 1 - \sqrt{\frac{ \mu \gamma}{9}}\right)^N [\| x^0 - x^*\|^2 + 2\gamma \eta^2 (f(x^0_f) - f(x^*))]
    \\&
    + \frac{81}{2}\beta \mu^{-2} \Big(\left( 1 + 2 /\beta \right) \cdot 2C_{1}\taumix \batchbound^{-1} + 208 C_{1} \taumix \log_2 M\Big) \cmax^2
    \\
    \leq&
    \exp\left( - \sqrt{\frac{\mu \gamma N^2}{9}}\right) [\| x^0 - x^*\|^2 + 2\gamma \eta^2 (f(x^0_f) - f(x^*))]
    \\&
    + \frac{81 \sqrt{\gamma}}{\mu^{3/2}} C_{1} \taumix \Big(1 + 104 \log_2 M\Big) \cmax^2\,.
\end{align*}
Substituting of $\eta = \sqrt{\tfrac{9}{\mu \gamma}}$ concludes the proof.\hfill$\square$
\end{proof}
\begin{corollary}[\textnormal{Step tuning for Theorem \ref{th:acc}}]
    Under the conditions of Theorem \ref{th:acc}, choosing $\gamma$ as 
    \begin{equation*}
        \gamma \lesssim \min\left\{\frac{\mu^3}{\delta^4\tau^2} ; \frac{1}{L} ; \frac{1}{\mu N^2}\ln^2\left( \frac{\mu^2 N[\| x^0 - x^*\|^2 + 18\mu^{-1} (f(x^0_f) - f(x^*))]}{ \tau\sigma^2}\right)\right\},
    \end{equation*}
    in order to achieve $\epsilon$-approximate solution (in terms of $\EE\bigl[\|x^{N} - x^*\|^2\bigr] \lesssim \epsilon$) it takes
    \begin{equation*}
        \Tilde{\mathcal{O}}\left(\left(\sqrt{\frac{L}{\mu}} + \frac{\tau\delta^2}{\mu^2}\right)\log\left(\frac{1}{\epsilon}\right) + \frac{\tau\sigma^2}{\mu^2 \epsilon}\right) \text{~oracle calls.}
    \end{equation*}
\end{corollary}
% \large{\textbf{Discussion}}

\section{Numerical experiments}
In this section, we present numerical experiments that compare the proposed method and the existing approaches for the problem of finding consensus in distributed network.
\vspace{-0.4cm}
\subsection{Problem formulation}
\vspace{-0.1cm}
Let us consider the next problem. Assume that we have $\{x_i\}_{i=1}^d$, where $x_i \in \mathbb{R}$. Also we get a communication network, where $i^{th}$ agent stores $x_i$. Moreover, the communication graph can be described as $G_k = (V, E_k)$, where the set of edges depends on the $k$ -- the current moment. The task is formulated as a consensus search, i.e., to find $\overline{x} = \frac{1}{d}\sum_{i=1}^dx_i$ -- the average value of the agents.

\noindent To formalize our problem, we introduce the Laplacian matrix of the graph $G_k$: $W_k = D_k - A_k$ (here $D_k$ is the diagonal matrix with degrees of nodes, $A_k$ -- adjacency matrix) and its properties:
\begin{enumerate}
    \item $[W_k]_{i,j} \neq 0$  if and only if $(i,j) \in E_k$ or $i=j$,
	%\item $\mW(q)  \ones_n= \zeros_n$, $\ones_n^\top \mW(q) = \zeros_n^\top$, where $\ones_n = (1,\ldots,1) \in \R^n$, $\zeros_n = (0,\ldots,0) \in \R^n$,
	\item $\ker W_k \supset  \left\{ (x_1,\ldots,x_d) \in \R^d : x_1 = \ldots = x_d \right\}$,
	\item $\text{range } W_k \subset \left\{(x_1,\ldots,x_d) \in\R^d : \sum_{i=1}^d x_i = 0\right\}$.
\end{enumerate}
If we consider $x = (x_1, \ldots, x_d)^\top$, then, because of second property, one can obtain 
\begin{align*}
    x_1 = \ldots = x_n \Leftrightarrow W_kx = 0.
\end{align*}
Moreover, it is known that
\begin{align*}
    W_kx = 0 \Leftrightarrow \sqrt{W_k}x = 0.
\end{align*}
Hence, the problem of finding the consensus on the moment $k$ can be reformulated as
\begin{align}
    \label{eq: problem-consensus}
    \min_{x \in \mathbb{R^d}} \left[f(x): = \frac{1}{2}x^\top W_k x\right].
\end{align}
It is important that the problem formulations \eqref{eq: problem-consensus} for each $k$ have the same optimal point $x^*$, which is equal to consensus.

\noindent The classic approaches to find a consensus is a gossip protocol \cite{bertsekas2015parallel}. In terms of problem \eqref{eq: problem-consensus} the method can be formulated as a gradient descent:
\begin{align*}
    x^{k+1} = x^k - \gamma W_k x^k = (1 - \gamma W_k)x^k.
\end{align*}
This iteration sequence gives the consensus since the third property is fulfilled -- it allows to keep the sum of coordinates of $x^k$ the same, preventing the departure from the desired optimal point.

\noindent As mentioned above, the problem changes over time as the set of edges specifying the communication system changes. This situation occurs quite often in practice -- when additional resources are available to improve the network, edges may be added to speed up processes, and in some system failures, communications between agents may be disconnected due to crashes and overloads. Therefore, it is natural to assume that the changes in the graphs $G_k$ occur according to the Markovian law, since the changes are confined only to the current state of the communication system.

\noindent Since for the problem \eqref{eq: problem-consensus} the gradient is equal to $W_kx$, we have 
\begin{align*}
    \norm{W_kx - \mathbb{E}(W_k)x}^2 &= \norm{W_kx - W_kx^* - \mathbb{E}(W_k)x + \mathbb{E}(W_k)x^*}^2 \\ &\leq \lambda_{max}^2(W_k - \mathbb{E}(W_k))\norm{x - x^*}^2,
\end{align*}
where $\mathbb{E}(W_k)$ is an expectation of Laplacian matrix of a graph $G_k$ taking into account the stochasticity responsible for the changes in the graph (more detailed description see later). Consequently, the considered problem satisfies Assumption \ref{as:bounded_markov_noise_UGE}, what means that the theoretical analysis of our paper is applicable to \eqref{eq: problem-consensus}.
\subsection{Setup}
In numerical experiments, we consider the problem described above on different topologies with certain Markovian stochastisity.

\noindent\textbf{Brief description.} We design the experiments in the following way. Suppose we have some starting, or base topology. Then we modify it according to some Markovian law, during which we cannot affect the base graph (i.e., discard edges from it). Based on these changes, we compare two methods: proposed and classic one.

\noindent \textbf{Topologies.} As a base topologies we consider two types of graphs -- cycle-graph and star-graph. For each starting network we conducted numerical experiments for problems with different dimensions: $10, \ 100, \ 1000$.

\noindent \textbf{Markovian stochasticity.} The network changes in time in the certain way. On each moment $k$ with probability $\frac{1}{2}$ the random edge can be added to the topology, but if it already exists in the graph, then nothing happens. At the same time, with the same probability the random edge can be removed from the network. Nevertheless, if this edge is in the base topology or communication topology does not contain this edge, we keep the graph in the same condition.
\vspace{-0.4cm}
\subsection{Results}
We performed numerical experiments with different base topologies (see Figures \ref{fig: cycle} and \ref{fig: star}) with $d = 10$ (see Figures \ref{fig: cycle_10}, \ref{fig: star_10}), $100$ (see Figures \ref{fig: cycle_100}, \ref{fig: star_100}) and $1000$ (see Figures \ref{fig: cycle_1000}, \ref{fig: star_1000}). As a result, the proposed method outperform the classic approach \cite{bertsekas2015parallel} showing a faster rate of convergence, especially for the high-dimension problem.
\begin{figure}[h]
     \centering
     \begin{subfigure}[b]{0.43\textwidth}
         \centering
          \includegraphics[width=\textwidth]{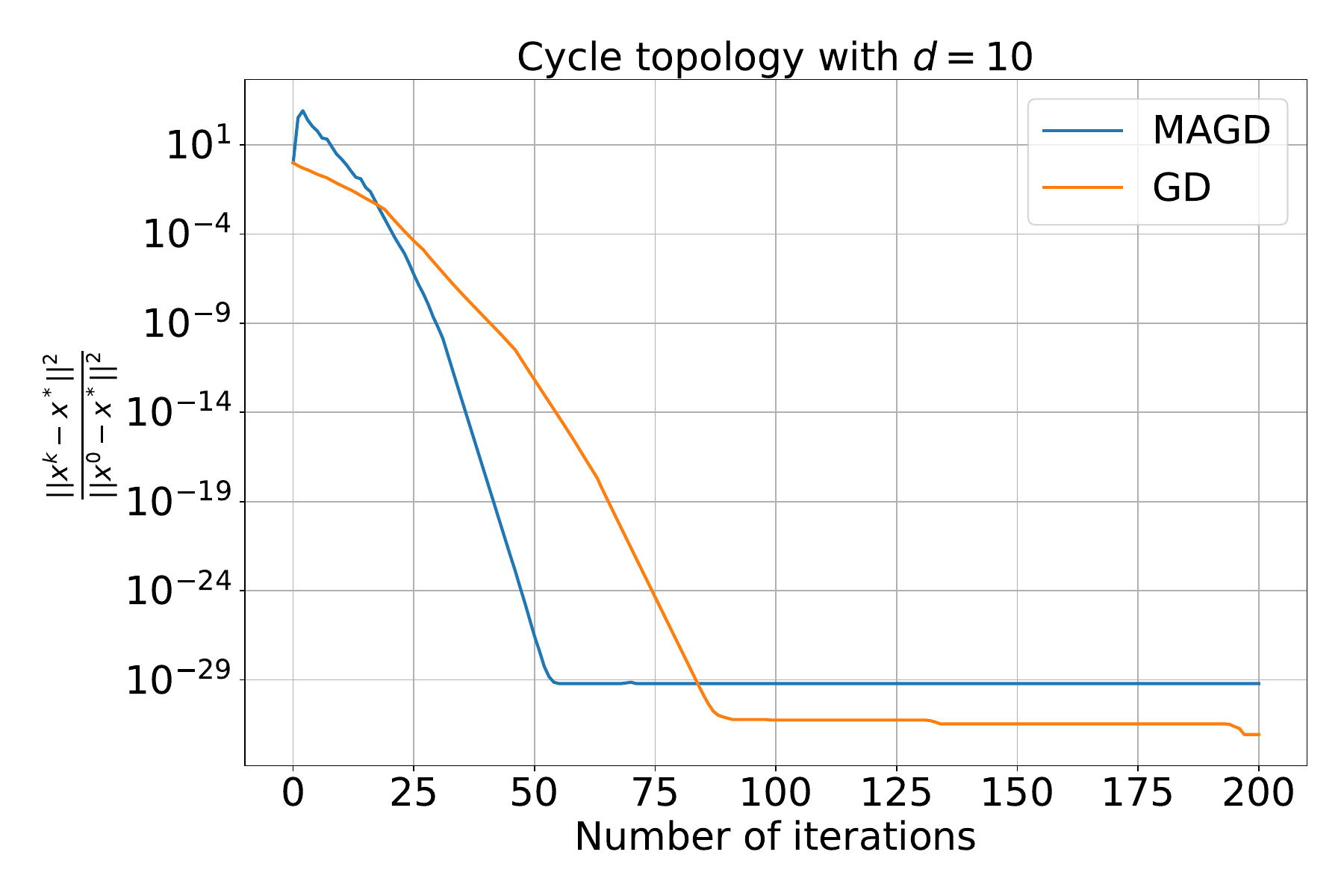}
         \caption{$d = 10$}
         \label{fig: cycle_10}
     \end{subfigure}
     \hfill
     \begin{subfigure}[b]{0.43\textwidth}
         \centering
         \includegraphics[width=\textwidth]{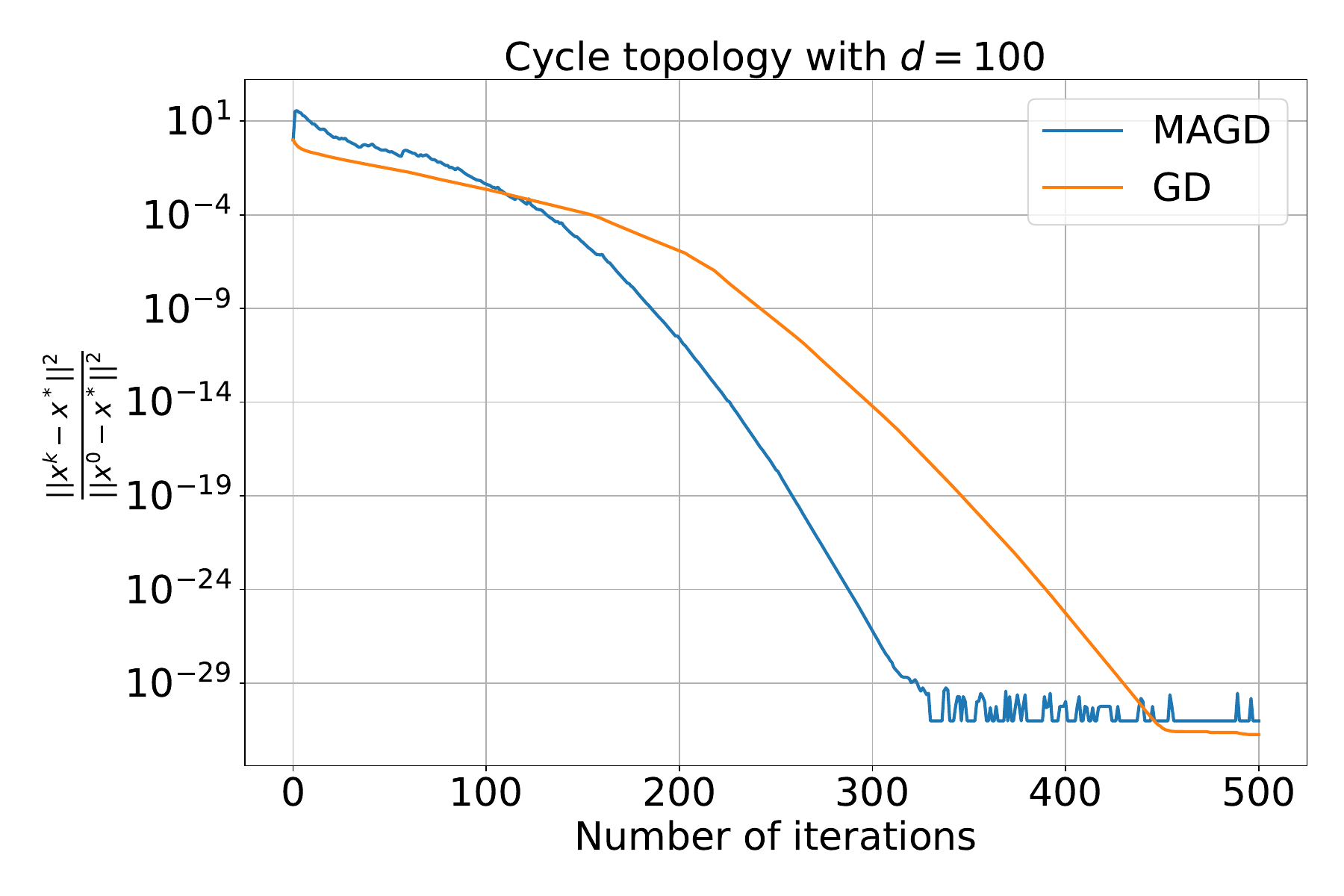}
         \caption{$d = 100$}
         \label{fig: cycle_100}
     \end{subfigure}
     \vfill
     \begin{subfigure}[b]{0.43\textwidth}
         \centering
         \includegraphics[width=\textwidth]{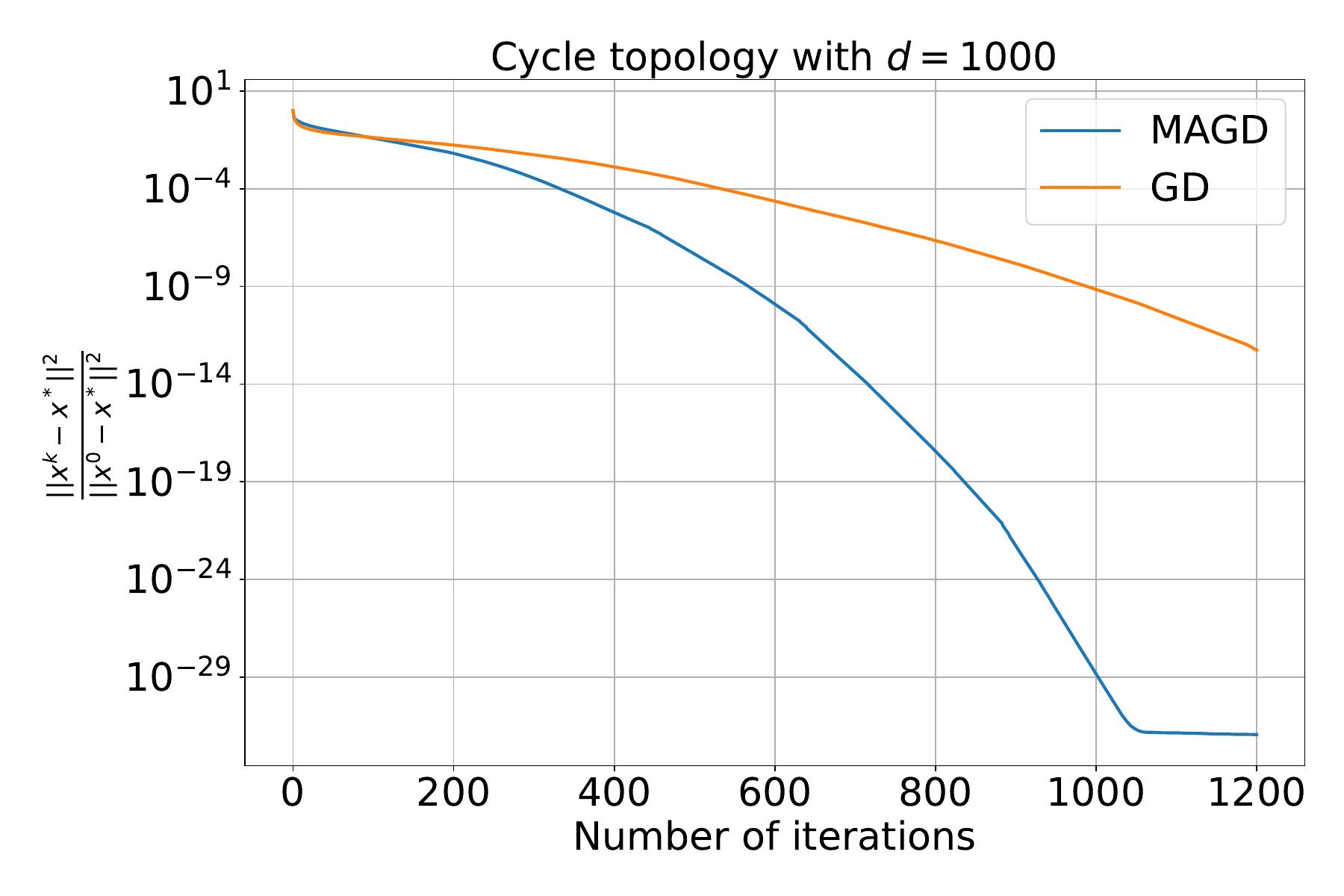}
         \caption{$d = 1000$}
         \label{fig: cycle_1000}
     \end{subfigure}
        \caption{Comparison of MAGD and GD for the consensus problem \eqref{eq: problem-consensus} on the cycle topology with different dimensions.}
        \label{fig: cycle}
\end{figure}
\begin{figure}[h]
     \centering
     \begin{subfigure}[b]{0.43\textwidth}
         \centering
         \includegraphics[width=\textwidth]{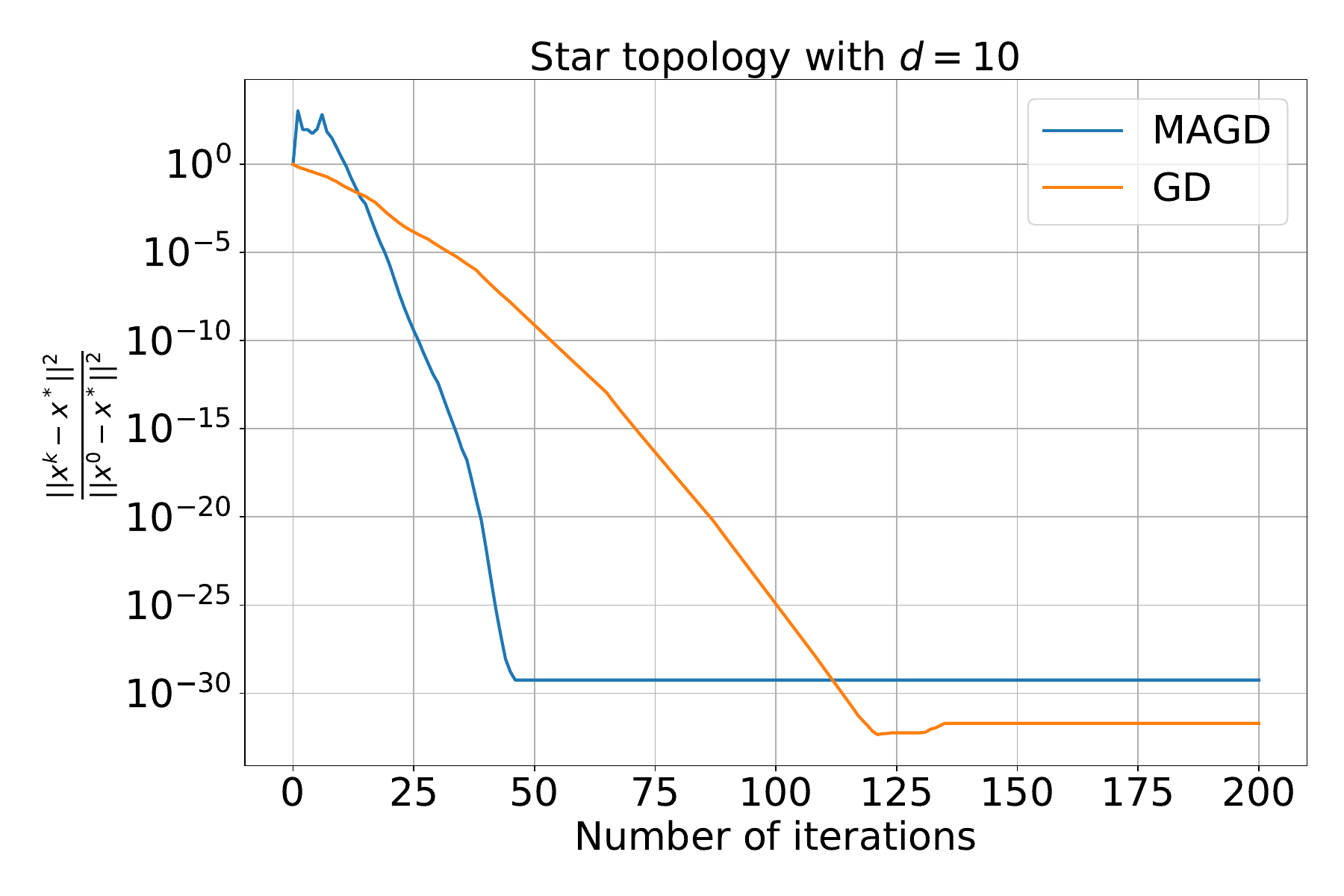}
         \caption{$d = 10$}
         \label{fig: star_10}
     \end{subfigure}
     \hfill
     \begin{subfigure}[b]{0.43\textwidth}
         \centering
         \includegraphics[width=\textwidth]{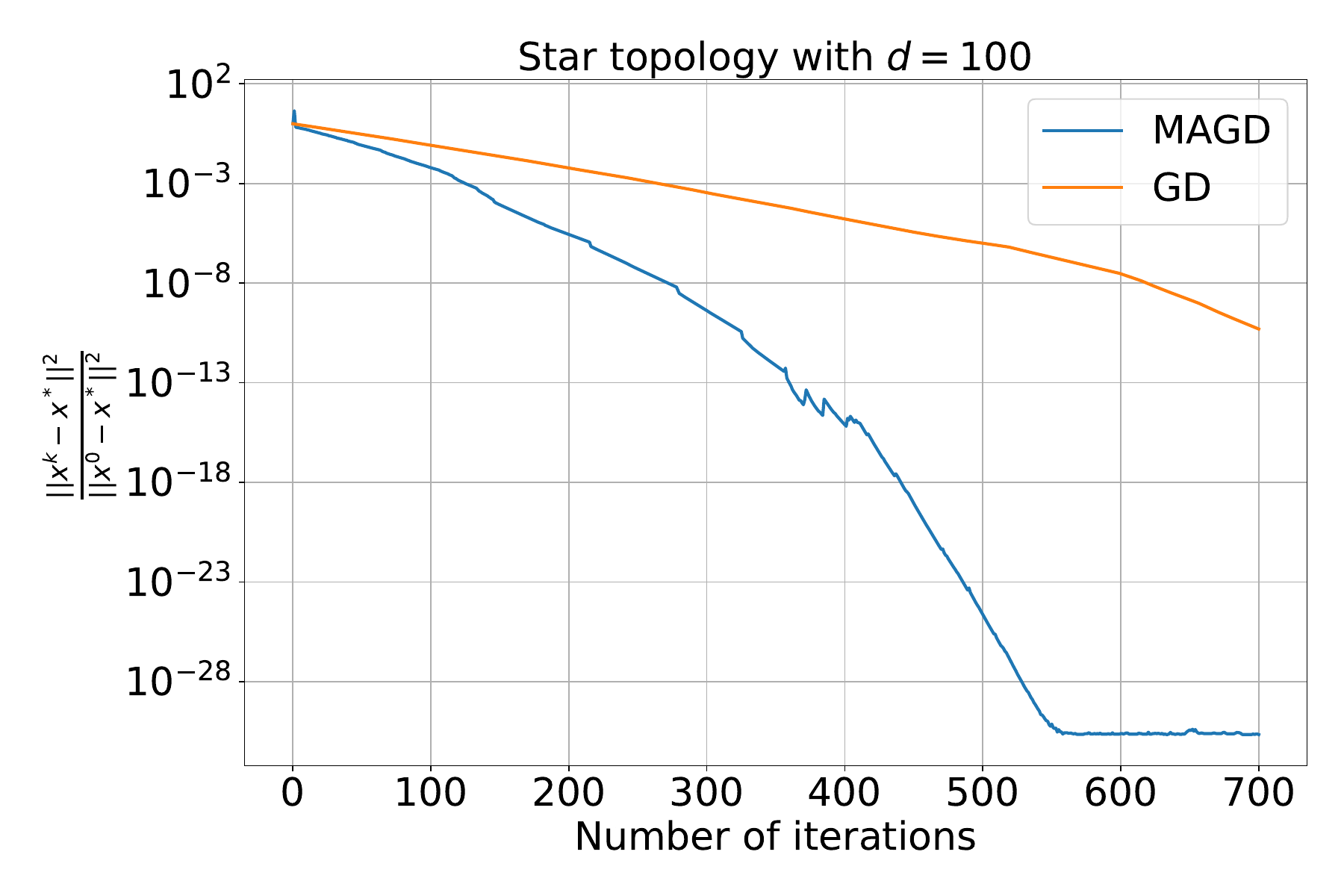}
         \caption{$d = 100$}
         \label{fig: star_100}
     \end{subfigure}
     \vfill
     \begin{subfigure}[b]{0.43\textwidth}
         \centering
         \includegraphics[width=\textwidth]{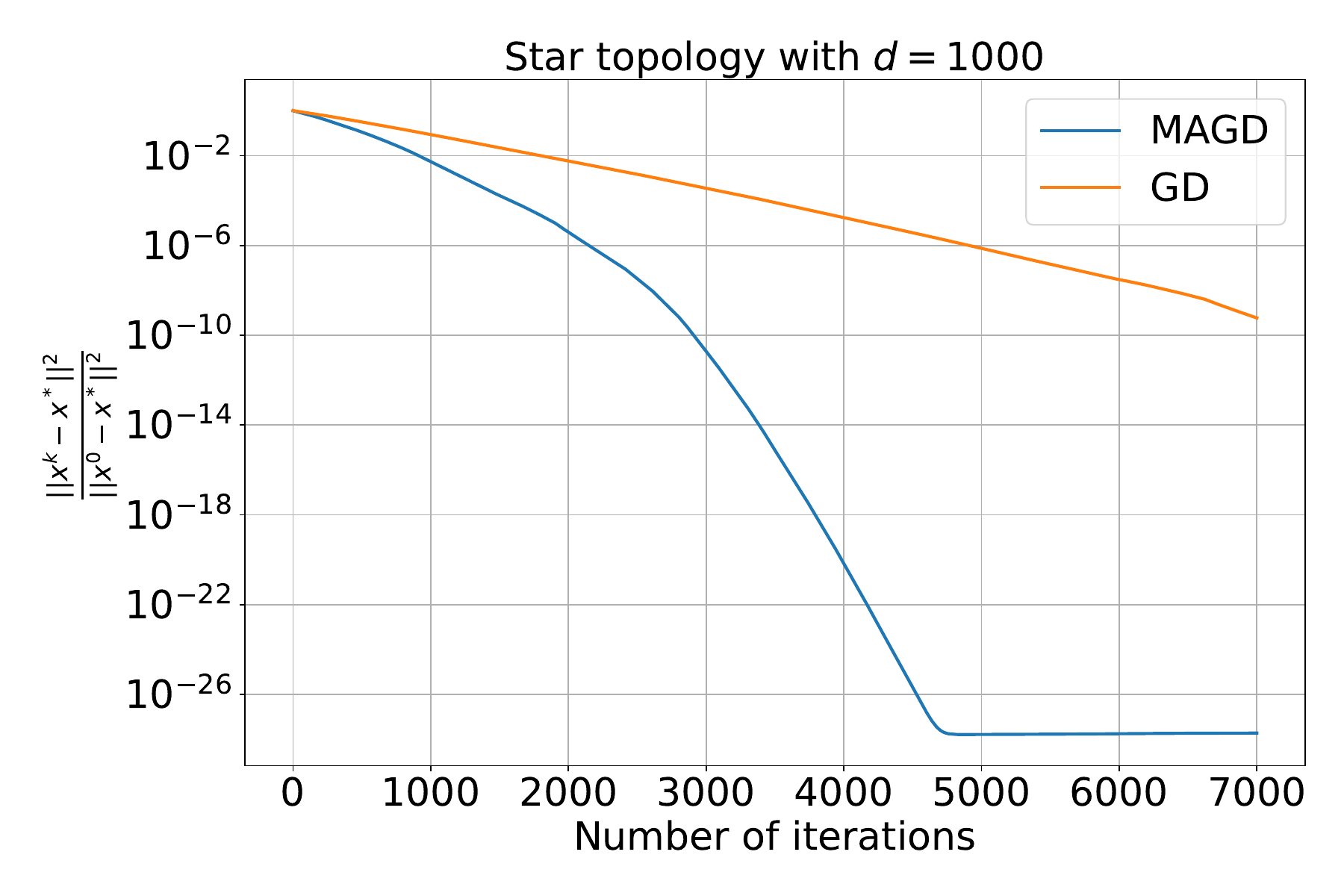}
         \caption{$d = 1000$}
         \label{fig: star_1000}
     \end{subfigure}
        \caption{Comparison of MAGD and GD for the consensus problem \eqref{eq: problem-consensus} on the star topology with different dimensions.}
        \label{fig: star}
\end{figure}

\end{document}